\documentclass[twoside]{article}
\usepackage{amssymb}
\usepackage{amsthm}
\usepackage{array}
\usepackage{amsfonts} 
\usepackage{amsmath}
\usepackage{bbm}
\usepackage{graphicx}
\usepackage{caption}
\usepackage{subcaption}
\usepackage{amsmath}
\usepackage{mathtools}
\usepackage{amsmath}
\usepackage{footnote}
\usepackage{enumerate} 
\usepackage{geometry} 
\usepackage{setspace}
\usepackage{hyperref}
\usepackage{paralist}
\usepackage{verbatim}

\theoremstyle{definition}

\theoremstyle{plain}
\newtheorem{thm}{Theorem}[section]

\newtheorem{remark}{Remark}[section]
\newtheorem{lemma}[thm]{Lemma}

\newtheorem{prop}[thm]{Proposition}

\title{Optimal dividend payments under a time of ruin constraint: Exponential claims}

\author{Camilo Hern\'andez\thanks{Department of Mathematics, Universidad de los Andes, Bogota, Colombia. Email address: {\tt mc.hernandez131@uniandes.edu.co}}\and Mauricio Junca\thanks{Department of Mathematics, Universidad de los Andes, Bogota, Colombia. Email address: {\tt mj.junca20@uniandes.edu.co}}}

\begin{document}

\maketitle

\begin{abstract}
We consider the classical optimal dividends problem under the Cram\'er-Lundberg model with exponential claim sizes subject to a constraint on the time of ruin. We introduce the dual problem and show that the complementary slackness conditions are satisfied, thus there is no duality gap. Therefore the optimal value function can be obtained as the point-wise infimum of auxiliary value functions indexed by Lagrange multipliers. We also present a series of numerical examples.
\end{abstract}

\section{Introduction}\label{Int}

One of the most studied models in actuarial science to describe the reserves process of an insurance company is the Cram\'er-Lundberg model. In this model the company faces claims whose arrivals follow a compound Poisson process and a constant premium is paid by the insured clients.

After the model was introduced, the probability of ruin of such a portfolio was among the principal interests in this field, see \cite{AsmuAlbr}. Nowadays, results about minimizing ruin probability considering reinsurance and investment in risky assets are proved by \cite{schmidli2002}. In \cite{Schmidli}, similar results for a discrete time version of the model and a diffusion approximation are shown. However, a process that does not end in ruin in a model exceeds every finite level, this is, the company lives an infinite period of time and accumulates an infinite amount of money, which is quite unrealistic in practice.

This idea motivated the study of the performance, instead of the safety aspect, of such portfolio. In 1957, Bruno de Finetti was interested in finding a way of paying out dividends in order to optimise the expected present value of the total income of the shareholders from time zero until ruin. This problem is commonly referred as de Finetti's problem, see \cite{Definetti}. As a result, researchers addressed the optimality aspect under more general and realistic assumptions which has turned out to be an abundant and ambitious field of research. See \cite{Schmidli} for results on this problem in the Cram\'er-Lundberg model and its diffusion approximation, and \cite{azcuemuler2005} when reinsurance is also considered.

Under exponential claims, see \cite{Schmidli}, the optimal dividends payment strategy to de Finetti's problem is known to be a \emph{barrier strategy}, this is, there exists a value $b$ such that the optimal strategy is to let $D_t=[ \bar{X}_t-b ]\vee 0$ where $\bar{X_t}$ denotes the maximum of the reserves process $X$ up to time $t$. However, the solution is not necessarily of this type in the general set up. In \cite{azcuemuler2005} the authors provided an example of claims distribution for which no barrier strategy is optimal.

However, there exists a trade-off between stability and profitability. Minimizing ruin probability means no dividends payment and profits tend to $0$. On the contrary, maximizing the dividends leads to a situation for which ruin is certain regardless of the initial capital, see \cite{buhlmann70}.

The idea of this work is to study a way to link two key concepts in the optimal dividends payment theory for the Cram\'er-Lundberg model: the profits and the time of ruin derived from a dividends payment strategy. Approaches in this direction have been already considered, see \cite{Hipp03} for a solution to the optimal dividends payment problem under a ruin constraint in discrete time and state setting. In \cite{Jostein03} the author addressed this matter introducing the concept of \emph{solvency constraints} and studied how they affected the optimal level of the barrier. In \cite{ThonAlbr} the authors consider the classical and diffusion approximation models and introduced a penalization on the time of ruin on the objective function, however an actual restriction on the time of ruin was not stated in this work. Finally, the model introduced in \cite{Grandits} took this link into account but in a different setting. More specifically, \cite{Grandits} proposed an iterative scheme to solve the problem of maximize the expected discounted consumption of an investor in finite time plus a penalisation on the level of the reserves at ruin for a given upper bound for the ruin probability.

This paper is organized as follows: Section \ref{problem} is dedicated to the formulation of the problem under consideration and to rewrite it using duality theory. In the following sections, under the exponential claim sizes assumption, we present the main results of this paper. We first focus on solving the dual problem in Section \ref{dual} and then show the absence of duality gap for this problem in Section \ref{gap}. This paper's contribution relies on both the solution of the optimal dividend payment problem under a constrain on the time of ruin, and the tools developed  in order to prove the duality gap of this problem to be zero. Section \ref{numerics} is dedicated to the presentation of numerical examples that illustrate different scenarios of the solution. In the last section we give conclusions of this study and present directions in which this work can be continued.

\section{Problem formulation}\label{problem}

In this paper we consider the Cram\'er-Lundberg risk model, in which the surplus follows the process:
\begin{equation*}
X_t= x_0 + ct - \sum_{i=1}^{N_t}Y_i,
\end{equation*}
where $N=(N_t)_{t\ge0}$ represents a homogeneous Poisson process with rate $\lambda>0$, modeling the claim occurrences. $\{ Y_i \}$ models the sequence of claim amounts, $\{ Y_i \} \overset{iid}{\sim} G(y)$, with $G(.)$ is a continuous distribution function on $[0,\infty)$. $\{Y_i\}$ is assumed to be independent of the claim occurrences process $N$. The deterministic components are the premium rate $c>0$ and the initial capital $x_0$. All of the above are defined in the same filtered probability space $(\Omega,\mathcal{F},(\mathcal{F}_t)_{t\geq0},\mathbb{P})$, with  $(\mathcal{F}_t)_{t\geq0}$ the filtration generated by the process $X$.

The insurance company is allowed to pay dividends which are model by the process $D=(D_t)_{t\geq0}$ representing the cumulative payments up to time $t$. A dividends process is called admissible if it is a non negative, non decreasing c\`adl\`ag process adapted to the filtration $(\mathcal{F}_t)_{t\geq0}$. Therefore, the surplus process under dividends process $D$ reads as
\begin{equation}\label{surplus}
X_t^D= x_0 + ct - \sum_{i=1}^{N_t}Y_i-D_t.
\end{equation}

Let $\tau^D$ denotes the time of ruin under dividends process $D$, i.e., $\tau^D=\inf\{t\geq0: X_t^D<0\}$. We require the dividends process not to lead to ruin, i.e., $D_{t-}-D_t\leq X_t^D$ for all $t$ and $D_t= D_{\tau^D}$ for $t\geq \tau^D$, so no dividends are paid after ruin. We call $\Theta$ the set of such processes. 

The company wants to maximise the expected value of the discounted flow of dividends payment along time, that is, to maximise $$\mathcal{V}^D(x_0):=\mathbb{E}_{x_0}\Big[ \int_0^{\tau^D}e^{-\delta t}d D_t\Big],$$ 
where the lifespan of the company will be determined by its ruin, see \cite{ThonAlbr}. Also, $\delta$ is the discount factor.

Finally, we add a restriction on the dividends process $D$ which we model by the equation:
\begin{equation}\label{Rest}
\mathbb{E}_{x_0} \Bigg[\int_0^{\tau^D}e^{-\delta s}ds\Bigg]\geq \int_0^T e^{-\delta s}ds \quad T\geq0 \text{ fixed.}
\end{equation}

The motivation behind such a constraint is that it imposes a restriction on the time of ruin. For simplicity, let's denote the right hand side of the restriction by $K_T$, i.e., $K_T:=\int_0^T e^{-\delta s}ds$. Note that $K_T \in [0,\frac{1}{\delta})$ and that the greater $K_T$ the greater $\tau^D$ must be. Evidently there are other restrictions that capture this effect in a more direct way, e.g., $\mathbb{E}[\tau^D] \geq K$. However there is no, to the best of our knowledge, a satisfactory technique to introduce such type of constraint in the model; additionally, as it will be clear in the following sections, the chosen functional form of the constraint fits in with the model in a smooth way. 

Combining all the above components we state the problem we aim to solve:
\begin{align*}\label{P1}
\tag{P1}
V(x_0):=\underset{D\in \Theta}\sup\quad \mathcal{V}^D(x_0) \quad \text{s.t.} \quad \mathbb{E}_{x_0} \Bigg[\int_0^{\tau^D}e^{-\delta s}ds\Bigg]\geq K_T \quad T \text{ fixed.}
\end{align*}

In order to solve this problem we use Lagrange multipliers to reformulate our problem. We first define the following for $ \Lambda \geq 0$ \begin{equation}\label{lagrangian}
\mathcal{V}_{\Lambda}^{D}(x_0):=\mathbb{E}_{x_0}\Bigg[ \int_0^{\tau^D}e^{-\delta t}d D_t +\Lambda \int_0^{\tau^D}e^{-\delta s}ds\Bigg]- \Lambda K_T.
\end{equation}

The following remark clears out the strategy we will follow in the remainder of the paper.
\begin{remark}
\begin{itemize}
\item Note that \eqref{P1} is equivalent to $\underset{D\in \Theta}\sup\,\, \underset{\Lambda\geq 0}\inf\,\,\mathcal{V}_{\Lambda}^{D}(x_0)$ since
$$\underset{\Lambda\geq 0}\inf\,\,\mathcal{V}_{\Lambda}^{D}(x_0)=\begin{cases} \mathcal{V}^D(x_0) &\mbox{if }\mathbb{E}_{x_0} \Bigg[\int_0^{\tau^D}e^{-\delta s}ds\Bigg]\geq K_T \\ 
-\infty & \mbox{otherwise }. \end{cases} $$
\item The dual problem of \eqref{P1}, is defined as 
\begin{equation}\label{D}
\tag{D}
\underset{\Lambda\geq 0}\inf\,\,\underset{D\in \Theta}\sup\,\, \mathcal{V}_{\Lambda}^{D}(x_0),
\end{equation}
which is always an upper bound for the primal \eqref{P1}.
\end{itemize}

The main goal is to prove that $\underset{D\in \Theta}\sup\,\, \underset{\Lambda\geq 0}\inf\,\,\mathcal{V}_{\Lambda}^{D}(x_0)= \underset{\Lambda\geq 0}\inf\,\,\underset{D\in \Theta}\sup\,\, \mathcal{V}_{\Lambda}^{D}(x_0)$.
\end{remark}

To solve \eqref{D} we note the last term of \eqref{lagrangian} does not depend on $D$ and is linear on $\Lambda$, therefore, we can focus on the first term on the right hand side of this equation and solve for fixed $\Lambda\geq0$
\begin{equation}\label{P2}
\tag{P2}
V_\Lambda(x_0):=\underset{D\in \Theta}\sup\,\, \mathcal{V}_{\Lambda}^{D}(x_0).
\end{equation}

For this problem, it is known that its solution must satisfy the following HJB equation, see Proposition 11 in \cite{ThonAlbr}

\begin{equation}\label{HJB}
\max \{\Lambda + cV'(x)+ \lambda \int_0^xV(x-y)dG(y)-(\lambda+\delta)V(x),1-V'(x)\}=0.
\end{equation}

As a first approach to this problem, we restrict ourselves assuming an exponential distribution for the claim sizes. In this scenario, we succeed in solving \eqref{P1} via \eqref{D} proving there is no duality gap. However, we have not yet approached the general problem for an arbitrary claim sizes distribution.

From now on we assume $\{Y_i\}\overset{iid}{\sim} Exp(\alpha)$. As a result \eqref{HJB} converts into
\begin{equation}\label{HJBexp}
\max \{\Lambda + cV'(x)+ \lambda \int_0^xV(x-y)\alpha e^{-\alpha y}dy-(\lambda+\delta)V(x),1-V'(x)\}=0,
\end{equation}

\section{Solution of \eqref{P2}}\label{dual}

For this problem \cite{ThonAlbr} proved that the optimal strategy corresponds to a barrier strategy and showed that the solution of \eqref{P2} is given by

\begin{thm}[Lemma 10 in \cite{ThonAlbr}]\label{Val}
Let $\Lambda\geq0$. Then, if $\Lambda\leq\frac{(\lambda+\delta)^2}{\alpha\lambda}-c$ the value function of \eqref{P2}  is given by
\[V_{\Lambda}(x)=x+ \frac{c+ \Lambda}{\lambda+\delta}+\frac{\Lambda}{\delta}(e^{-\delta T}-1).\]
Otherwise, there exists unique $b_\Lambda>0$ such that
\begin{equation}
V_{\Lambda}(x)=\begin{cases}
x-b_\Lambda+\frac{\alpha c-\lambda-\delta}{\alpha \delta}+\frac{\Lambda}{\delta}e^{-\delta T} &\text{if}\quad x\geq b_\Lambda \\
 C_1e^{r_1x}+C_2e^{r_2x}+\frac{\Lambda}{\delta}e^{-\delta T}  &\text{if}\quad x\leq b_\Lambda
\end{cases},
\end{equation}
Where $r_1, r_2$ are roots of the polynomial
\begin{equation}\label{poli}
 p(R)=c R^2+ (\alpha c -(\lambda+\delta)) R - \alpha \delta.
\end{equation}
$C_1, C_2$ are given by
\begin{align*}\label{const0}
C_2=-\frac{\alpha+r_2}{\alpha}\Bigg[\frac{\alpha}{\alpha+r_1}C_1+\frac{\Lambda}{\delta}\Bigg],\\
C_1=\frac{(\alpha+r_1)(\alpha\delta+e^{r_2b}\Lambda r_2(\alpha+r_2))}{\alpha\delta(e^{r_1b}r_1(\alpha+r_1)-e^{r_2b}r_2(\alpha+r_2))}.
\end{align*}
and the value of the optimal barrier level $b_\Lambda$ is given by the expression 
\begin{equation}\label{barrera}
(r_2-r_1)(\alpha+r_1)(\alpha+r_2)\Lambda =-r_1 e^{-r_2b}(r_1(\lambda+\delta)+ \alpha \delta)+r_2e^{-r_1b}(r_2(\lambda+\delta)+\alpha\delta).
\end{equation}
\end{thm}

Let us define the critical value $$\bar{\Lambda}=\frac{(\delta+\lambda)^2}{\alpha \lambda}-c.$$
\begin{remark}\label{critcrem}
Note that for $0\leq\Lambda\leq\bar{\Lambda}$ the optimal barrier strategy is $b=0$.
\end{remark}

The following proposition will be essential in the next section to prove the main result of this contribution.

\begin{prop}\label{blambda1to1}
\begin{enumerate}[(i)]
\item\label{caso1}  If $\bar{\Lambda}\geq 0$  for each $b> 0$, there exists a unique $\Lambda>\bar{\Lambda}$ such that $b$ is the optimal barrier dividend strategy for \eqref{P2} with $\Lambda$.
\item\label{caso2} If $\bar{\Lambda}< 0$, there exists $b_0>0$ such that for each $b\geq b_0$, there exists a unique $\Lambda\geq0$ such that $b$ is the optimal barrier dividend strategy for \eqref{P2} with $\Lambda$.
\end{enumerate}
\end{prop}
\begin{proof}
Note that equation \eqref{barrera} defines a map $\Lambda:[0,\infty)\rightarrow \mathbb{R}$. Taking derivative with respect to $b$ we obtain
\begin{equation}\label{dbdl}
\frac{d \Lambda(b)}{db}=\frac{r_1 r_2 [e^{-r_2b}(r_1(\lambda+\delta)+ \alpha \delta)- e^{-r_1b}(r_2(\lambda+\delta)+\alpha\delta)]}{(r_2-r_1)(\alpha+r_1)(\alpha+r_2)}
\end{equation}
It can be easily shown that both roots of the characteristic polynomial are real and non zero, furthermore, one of them is positive and the other is negative. The first follows since $[\alpha c -(\lambda+\delta)]^2+4c\alpha \delta >0$ and the second since $-[\alpha c -(\lambda + \delta)] +\sqrt{ [\alpha c -(\lambda+\delta)]^2+4c\alpha \delta} >0$, $-[\alpha c -(\lambda + \delta)] -\sqrt{ [\alpha c -(\lambda+\delta)]^2+4c\alpha \delta} <0$. Now, for the negative root, let's say $r_2$, we have $\alpha + r_2>0$ (see Lemma \ref{alphaandroot}). This leaves us with the case $r_1>0>r_2$, in which (\ref{dbdl}) is strictly positive as both the numerator and denominator are negative and therefore the map is injective. Furthermore, using the expressions for $r_1$ and $r_2$ in \eqref{barrera} we can show that
\begin{equation}\label{lambdacrit}
\Lambda(0)=\frac{(\lambda+\delta)^2-\alpha \lambda c}{\alpha\lambda}=\bar{\Lambda}.
\end{equation}
Hence, if  $\Lambda(0)<0$ and there exists $b_0>0$ (and solution of \eqref{bar0}) that satisfies \eqref{caso2} since $\Lambda\rightarrow\infty$ when $b\rightarrow\infty$. \eqref{caso1} follows from \eqref{lambdacrit}.
\end{proof}

Figure \ref{lambdab} shows values of $b_\Lambda$ derived from equation \eqref{barrera} for $\lambda=1,\, c=1.3,\, Y_i\sim Exp(1),\, $and$\,  \delta=0.1$. Note that this values fall in the second case of Proposition \ref{blambda1to1}.

\begin{figure}[h!]
\centering
\includegraphics[scale=.9]{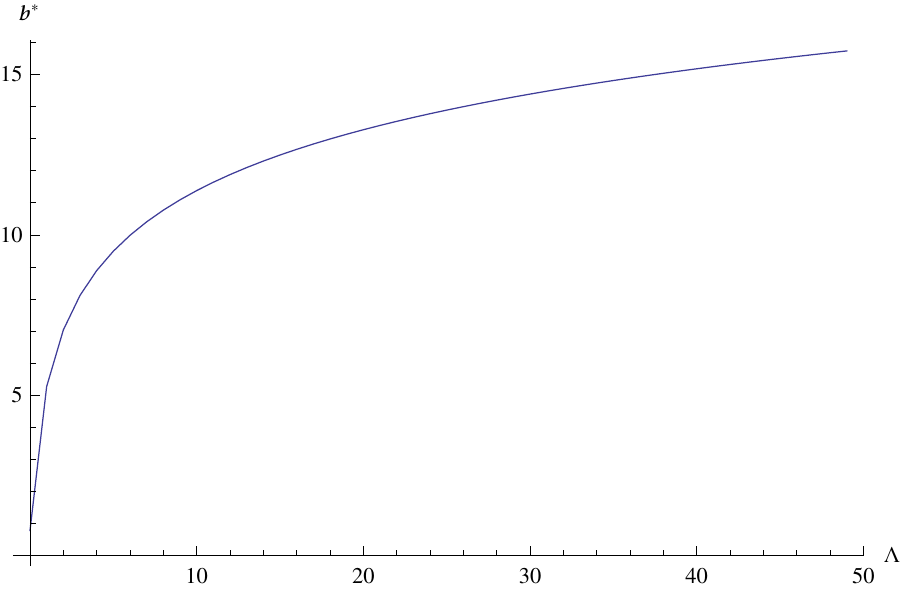}
\caption{Optimal barrier}\label{lambdab}
\end{figure}

\section{Solution of \eqref{P1}}\label{gap}

Let $X_t^{b}$ be the surplus process under dividend barrier strategy with level $b$ denoted by $D^b$. Let $\tau^b$ be the time of ruin using such strategy, i.e., $\tau^b:=\inf\{t:X_t^{b}<0\}$. In order to find out the solution to \eqref{P1} we will need the following proposition. 

\begin{prop}\label{optimalpair}
For each $x_0\geq0$ there exists $H_{x_0}\geq0$ such that if $0\leq T<H_{x_0}$ there exists $(\Lambda^*,b^*)$ that satisfies:
\begin{enumerate}[(i)]
\item $\Lambda^*\geq0$ and $b^*$ is the optimal barrier for \eqref{P2} with $\Lambda^*$ and initial value $x_0$,
\item\label{cond2} $\mathbb{E}_{x_0}\left[\int_0^{\tau^{b^*}}e^{-\delta s}ds\right]\geq K_T$ and
\item\label{cond3} $\Lambda^*\left(\mathbb{E}_{x_0}\left[\int_0^{\tau^{b^*}}e^{-\delta s}ds\right] -K_T\right)=0$.
\end{enumerate}
\end{prop}

\begin{proof}
Let $x_0\geq0$ fixed. Consider the following IDE problem:
\begin{align*}\label{A1}
\tag{A1}
& \mathcal {A}(\varPsi_b)(x)-\delta \varPsi_b(x)=-1\\
&\varPsi_b'(b)=0\\
&\varPsi_b \in C^1[0,b]
\end{align*}
where $\mathcal{A}(f)(x)$ is $cf'_b(x)+ \lambda \int_0^xf(x-y)\alpha e^{-\alpha y}dy-\lambda f(x)$, the infinitesimal generator of the surplus process $X_t$. It can be shown that for $0\leq x\leq b$
$$\varPsi_b(x)= \frac{1}{\delta} + C_1e^{r_1x}+C_2e^{r_2x}$$
with
$$C_1=\frac{(\alpha+r_1)(\alpha+r_2)r_2 e^{r_2b} }{\alpha\delta[r_1e^{r_1b}(\alpha+r_1)-r_2 e^{r_2b}(\alpha+r_2)]}$$
and
$$C_2= -\frac{(\alpha+r_2)^2 r_2 e^{r_2b}}{\alpha\delta(e^{r_1b}r_1(\alpha+r_1)-e^{r_2b}r_2(\alpha+r_2))}- \frac{(\alpha+r_2)}{\alpha\delta}$$
is solution of \eqref{A1}, where $r_1,\, r_2$ denote the roots of \eqref{poli}. Extend $\varPsi_b$ to $\mathbb{R}$ so that $\varPsi_b(x)=0,\, x< 0$ and $\varPsi_b(x)=\varPsi_b(b), x\geq b$. Using Dynkin's formula and the Optional Stopping Theorem we obtain that for $0\leq x\leq b$
\begin{align*}
\mathbb{E}_{x}[e^{-\delta \tau^b} \varPsi_b(X_{\tau^b}^{b})]&=\varPsi_b(x)+ \mathbb{E}_{x}\bigg[\int_0^{\tau^b} e^{-\delta s}[\mathcal {A}(\varPsi_b(X_{s}^{b}))-\delta \varPsi_b(X_s^{b})]ds \\
&+ \sum_{0\le s\le t, \bigtriangleup D^b \neq 0}e^{-\delta s}\big[ \varPsi_b(X_s^{b})- \varPsi_b(X_{s-}^{b})\big] + \int_0^{\tau^b}e^{-\delta s} \varPsi'_b(X_s^{b})d \bar{D}_s^b \bigg]\\
&=\varPsi_b(x)+ \mathbb{E}_{x}\bigg[\int_0^{\tau^b} e^{-\delta s}[\mathcal {A}(\varPsi_b(X_{s}))-\delta \varPsi_b(X_s)]ds \bigg]\\ 
&+\mathbb{E}_{x}\bigg[\int_0^{\tau^b} e^{-\delta s}c\varPsi'_b(b) \mathbf{1}_{\{X_s=b\}} ds\bigg],
\end{align*}
where $\bar{D}^b$ denotes the continuous part of the control $D^b$. For the last equality we used that the continuous part of the control consists only of $c$ at the moment at which $X_s^{b}=b$ and that the control has no jumps. Therefore, for $\varPsi_b(x)$, the extended solution of \eqref{A1}, we get that 
$$\varPsi_b(x)=\mathbb{E}_{x}\left[\int_0^{\tau^{b}}e^{-\delta s}ds\right].$$
 Define, 
\[\hat{\varPsi}(x):=\lim_{b\to \infty}\varPsi_b(x)=\frac{1}{\delta}-\frac{\alpha+r_2}{\alpha \delta}e^{r_2 x}.\]
Let $H_{x_0}:= -\frac{1}{\delta}(\log(\frac{\alpha + r_2}{\alpha})+r_2x_0)$ and suppose $0\leq T<H_{x_0}$.  Recall $K_T=\frac{1-e^{-\delta T}}{\delta}$ and note that $K_{H_{x_0}}=\hat{\varPsi}(x_0)$, so $K_T<\hat{\varPsi}(x_0)$. Then we have the following to cases:
\begin{enumerate}
\item Suppose $\bar{\Lambda}\geq 0$. If $\varPsi_{0}(x_0)\geq K_T$ then \eqref{cond2} is satisfied and the barrier $b^*=0$ is optimal for \eqref{P2} with $\Lambda^*=0$ by Remark \ref{critcrem}. On the other hand, if $\varPsi_{0}(x_0)< K_T < \hat{\varPsi}(x_0) $ Lemma \ref{Phiincreasinb} guarantees the existence of a unique $b^*>0$ such that $\varPsi_{b^*}(x_0)=K_T$. In this later case, \eqref{caso1} of Proposition \ref{blambda1to1}  guarantees the existence of a unique $\Lambda^*$ for which $b^*$ is optimal for \eqref{P2} with $\Lambda^*$. In both cases we have \eqref{cond3}.
\item Suppose $\bar{\Lambda}<0$. If $\varPsi_{b_0}(x_0)\geq K_T$  then \eqref{cond2} is satisfied and the unconstrained problem satisfies the restriction. Therefore, $b^*=b_0$ is optimal for \eqref{P2} with $\Lambda^*=0$. If $\varPsi_{b_0}(x_0)< K_T < \hat{\varPsi}(x_0) $ just as before we know there exists a unique $b^*$ such that $\varPsi_{b^*}(x_0)=K_T$. By \eqref{caso2} of Proposition \ref{blambda1to1} there exists a unique $\Lambda^*$ for which $b^*$ is optimal for \eqref{P2} with $\Lambda^*$. In both cases we also have \eqref{cond3}. 
\end{enumerate}
\end{proof}

As a consequence we have the main theorem:
\begin{thm}\label{strongduality}Let $x_0\geq 0$, $T\geq0$ and $V(x_0)$ be the optimal solution to \eqref{P1}. Then
\begin{enumerate}[(i)]
\item\label{th1} $ V(x_0)\leq \underset{\Lambda\geq 0}\inf\,\,V_{\Lambda}(x_0) $ and
\item\label{th2} $\underset{\Lambda\geq 0}\inf\,\,V_{\Lambda}(x_0) \leq V(x_0)$.
\end{enumerate}
Therefore, $\underset{\Lambda\geq 0}\inf\,\,V_{\Lambda}(x_0)=V(x_0)$.
\end{thm}

\begin{proof}
Fix $x_0\geq 0$. Condition \eqref{th1} is satisfied since $\underset{\Lambda\geq 0}\inf\,\,V_{\Lambda}(x_0)$ is the dual problem of \eqref{P1}. To verify condition \eqref{th2} we have the following cases:
\begin{enumerate}[(i)]
\item \underline{$T< H_{x_0}$}: By Proposition \ref{optimalpair} there is a pair $(\Lambda^*,b^*)$ such that
\begin{align*}
\underset{\Lambda\geq 0}\inf\,\,V_{\Lambda}(x_0)&\leq V_{\Lambda^*}(x_0)\\
&= \mathbb{E}_{x_0}\Bigg[ \int_0^{\tau^{b^*}}e^{-\delta t}d D^{b^*}_t + \Lambda^* \int_0^{\tau^{b^*}}e^{-\delta t} dt \Bigg] - \Lambda ^*K_T\\
&= \mathbb{E}_{x_0}\Bigg[ \int_0^{\tau^{b^*}}e^{-\delta t}d D^{b^*}_t\Bigg]\\
&\leq V(x_0),
\end{align*}
where the last inequality follows since the barrier strategy $b^*$ satisfies \eqref{Rest}.
\item \underline{$T=H_{x_0}$}: In this case $K_T= \hat{\varPsi}(x_0)$ and by Lemma \ref{limit} 
$$\Lambda\left(\mathbb{E}_{x_0}\left[\int_0^{\tau^{b_\Lambda}}e^{-\delta s}ds\right]- K_T\right)\rightarrow 0\quad \text{as} \quad \Lambda\rightarrow \infty.$$
Also $\mathbb{E}_{x_0}\left[ \int_0^{\tau^{b_{\Lambda}}}e^{-\delta t}d D^{b_\Lambda}_t\right]\rightarrow 0$  as $\Lambda\rightarrow \infty$ since $b_\Lambda\rightarrow\infty$.
Therefore, since $V_0(x_0)\geq0$ and $V_\Lambda(x_0)$ is convex in $\Lambda$ (it is the supremum of linear functions) we obtain that
$$\underset{\Lambda\geq 0}\inf\,\,V_{\Lambda}(x)= 0 \leq V(x_0).$$ 
\item \underline{$T> H_{x_0}$}: In this case $K_T> \hat{\varPsi}(x_0)$, therefore for all $b\geq0$ it holds that $\mathbb{E}_{x_0}\left[\int_0^{\tau^{b}}e^{-\delta s}ds\right] < \hat{\varPsi}(x_0)< K_T$. From this, one can deduce there exists $\epsilon>0$ such that $\Lambda\left(\mathbb{E}_{x_0}[\int_0^{\tau^{b_\Lambda}}e^{-\delta s}ds] -K_T\right)<-\Lambda \epsilon$. Letting $\Lambda \rightarrow\infty$ we obtain $\underset{\Lambda\geq 0}\inf\,\,V_{\Lambda}(x)=-\infty\leq V(x_0)$.
\end{enumerate}
\end{proof}

\section{Numerical examples}\label{numerics}

In this section we illustrate with more detail the cases that came up in the proof of Proposition \ref{optimalpair} and Theorem \ref{strongduality}. As presented in the previous section to obtain the optimal value function of \eqref{P1} we showed a pair $(b^*,\Lambda^*)$ that certified strong duality. To do so, we consider several cases depending on the initial value $x_0$ and $T$. We will continue to assume the following parameter values:  $ \lambda=1,\, c=1.3,\, Y_i\sim exp(1),\, $and$\,  \delta=0.1$. In each case we will show two graphs. Graphs on the left show $\varPsi_b(x)$ for different values of $b$ and graphs on the right show $V_\Lambda(x_0)$ for different values of $\Lambda$.

For the first case, choose $T$ and $x_0$ so that $K_T$ lies bellow $\varPsi_{b_0}(x_0)$. In this situation we know that the unconstrained solution satisfies the restriction. With such values the plot of $V_\Lambda(x_0)$ for different values of $\Lambda$ illustrate that the minimum is attained at $\Lambda^*=0$, see Figure \ref{figinactive}. The optimal solution is $V_0(x_0)$ and the optimal barrier is $b^*=b_0=0.8$.
\begin{figure}[h!]
   \begin{subfigure}[b]{.46\linewidth}
      \includegraphics[width=1\textwidth]{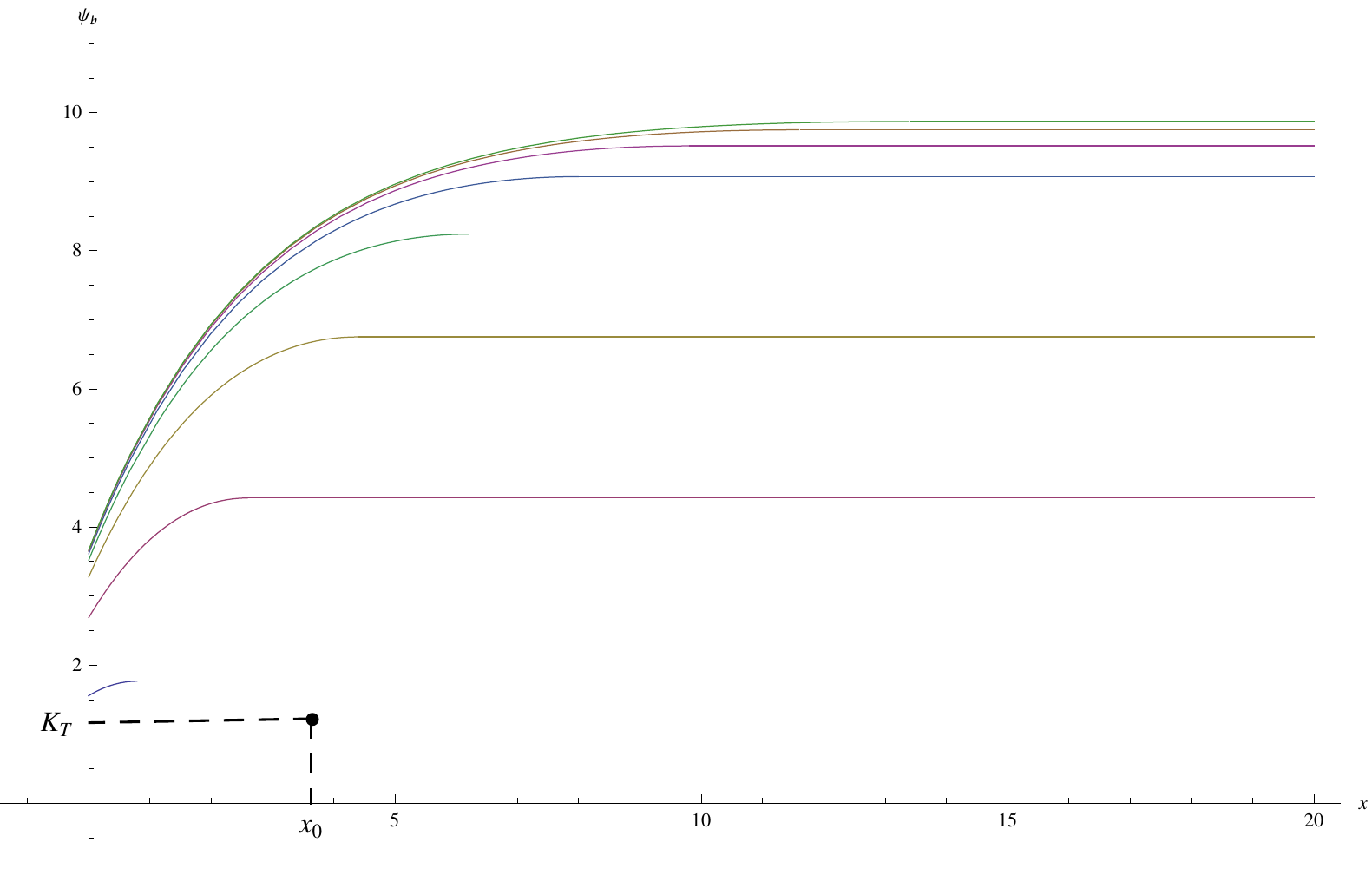}
   \end{subfigure}
\begin{subfigure}[b]{.46\linewidth}
      \includegraphics[width=1\textwidth]{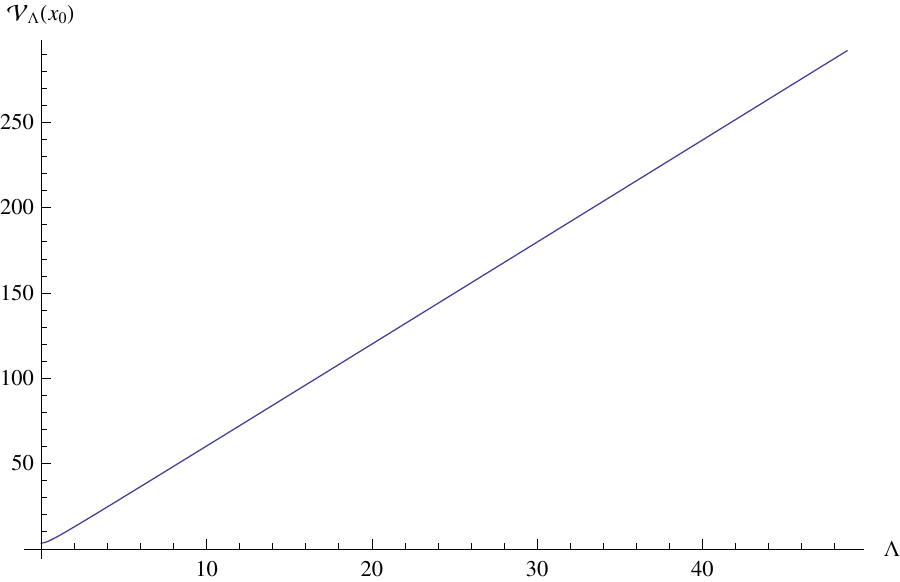}
   \end{subfigure}
      \caption{Inactive constraint.}\label{figinactive}
\end{figure}

In the second case, let $T$ and $x_0$ have values such that $K_T$ lies between $\hat{\varPsi}(x_0)$ and $\varPsi_{b_0}(x_0)$. With such values the plot of $V_\Lambda(x_0)$ for different values of $\Lambda$ reflects the existence of a minimum $\Lambda^*>0$. To find it, find $b^*$ that satisfies $\varPsi_{b^*}(x_0)=K_T$  and use Proposition \ref{blambda1to1} to get $\Lambda^*$, see Figure \ref{figactive}. The optimal solution is $V_{\Lambda^*}(x_0)$ and the optimal barrier is $b^*$.
\begin{figure}[h!]
   \begin{subfigure}[b]{.46\linewidth}
      \includegraphics[width=1\textwidth]{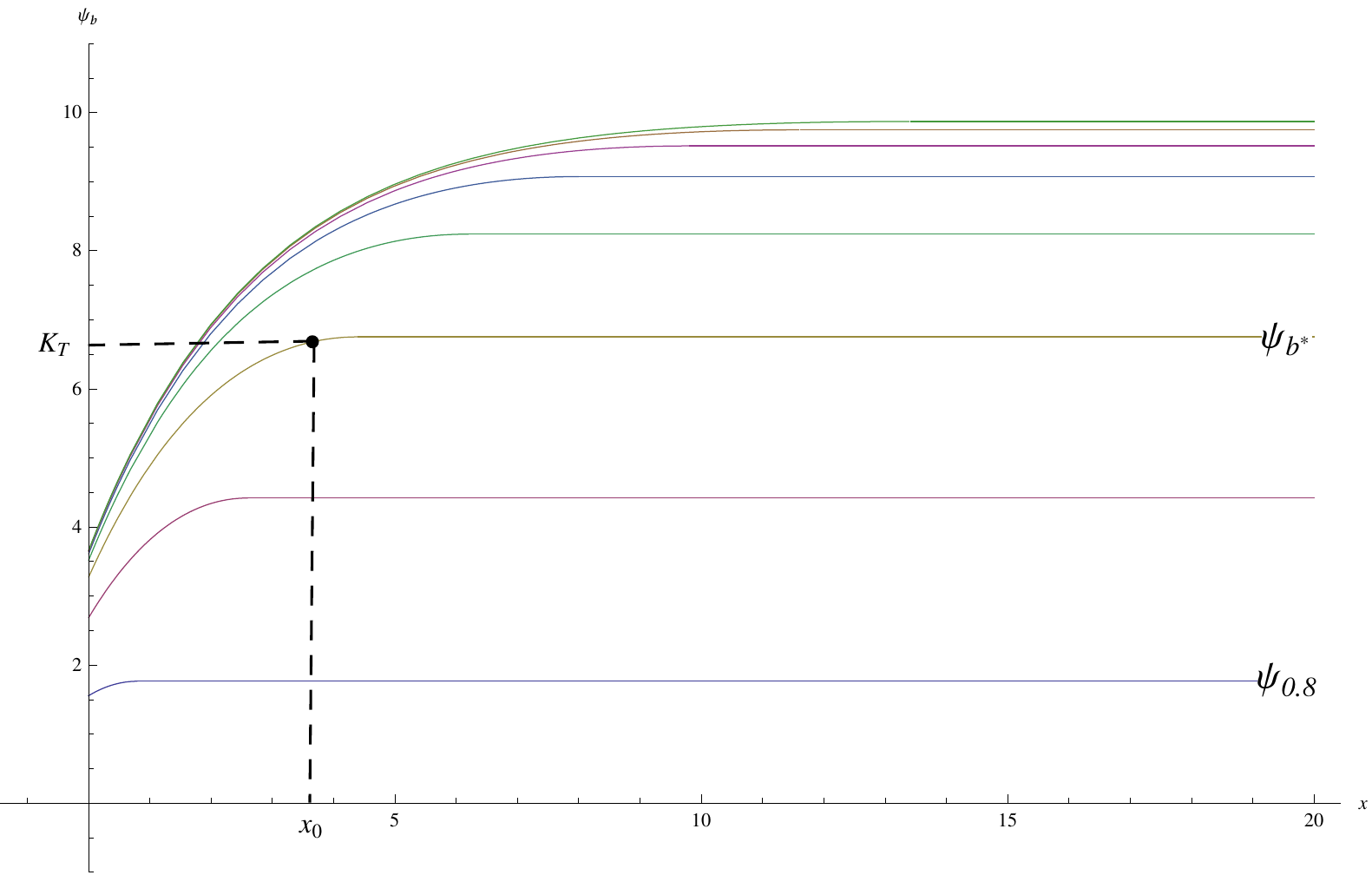}
   \end{subfigure}
\begin{subfigure}[b]{.46\linewidth}
      \includegraphics[width=1\textwidth]{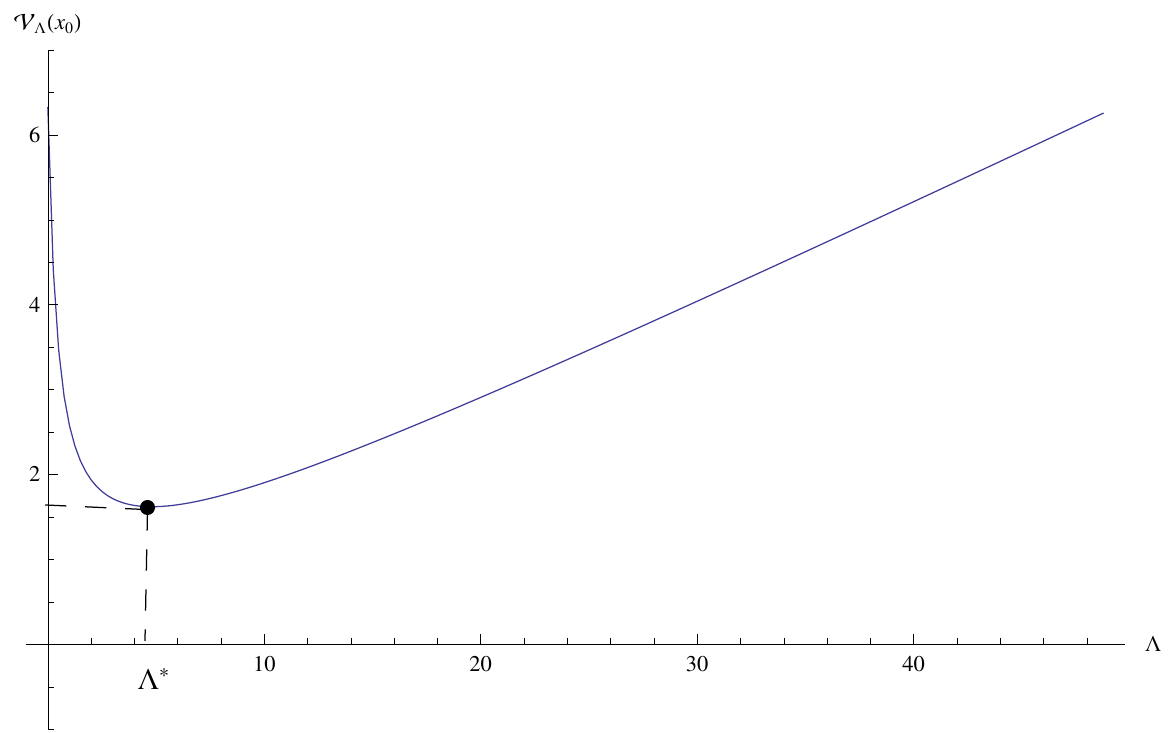}
   \end{subfigure}
      \caption{Active constraint.}\label{figactive}
\end{figure}
\clearpage
Now, let $T$ and $x_0$ have values such that $K_T=\hat{\varPsi}(x_0)$. With such values the plot of $V_\Lambda(x_0)$ for different values of $\Lambda$ shows that the minimum is attained at $\infty$ with a value of 0 see Figure \ref{figfrontera}. In this particular case we conclude that $V(x_0)=0$ so that for \eqref{P1} the optimal strategy is to do nothing. 
\begin{figure}[h!]
   \begin{subfigure}[b]{.49\linewidth}
      \includegraphics[width=1\textwidth]{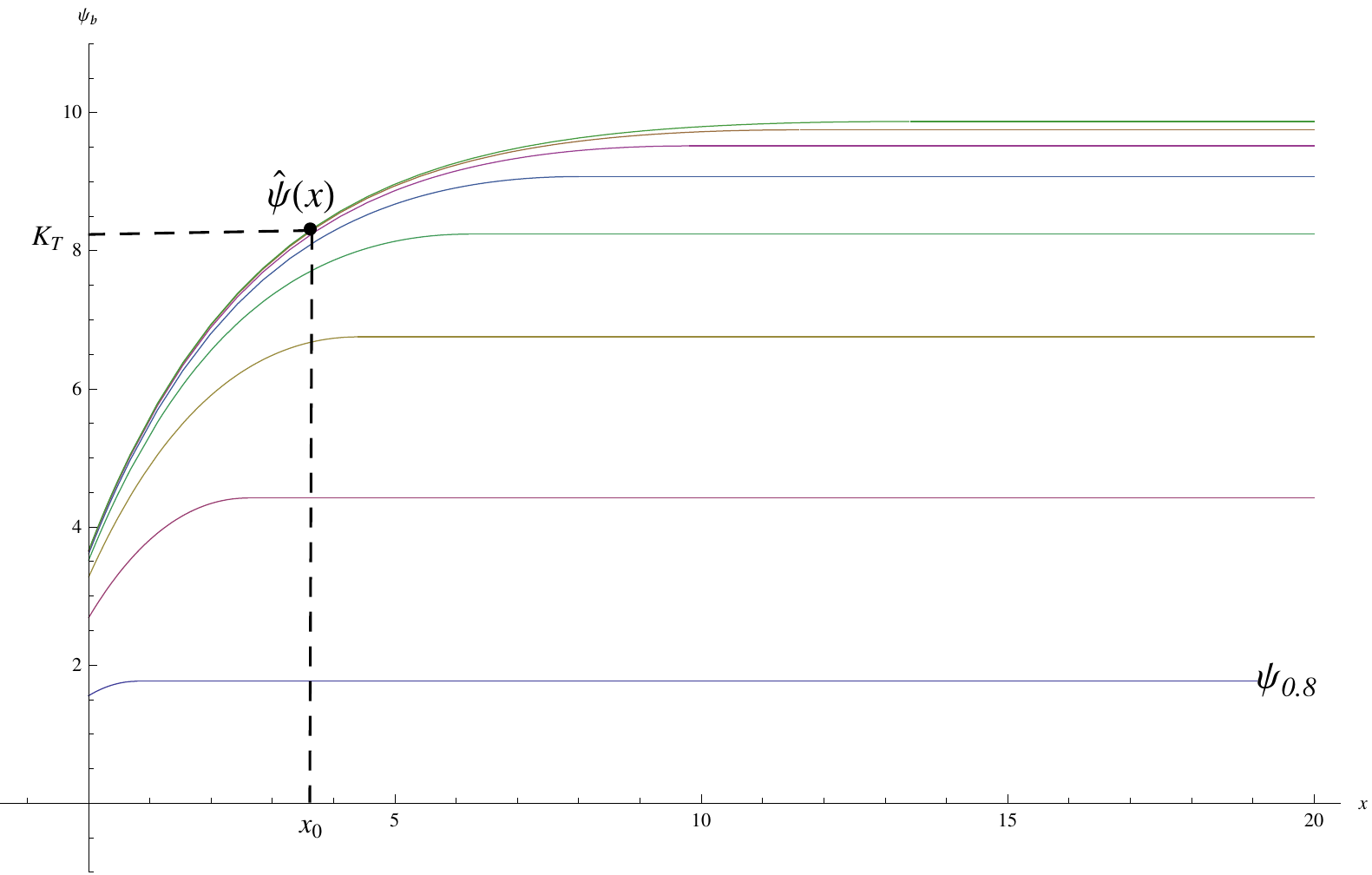}
   \end{subfigure}
\begin{subfigure}[b]{.49\linewidth}
      \includegraphics[width=1\textwidth]{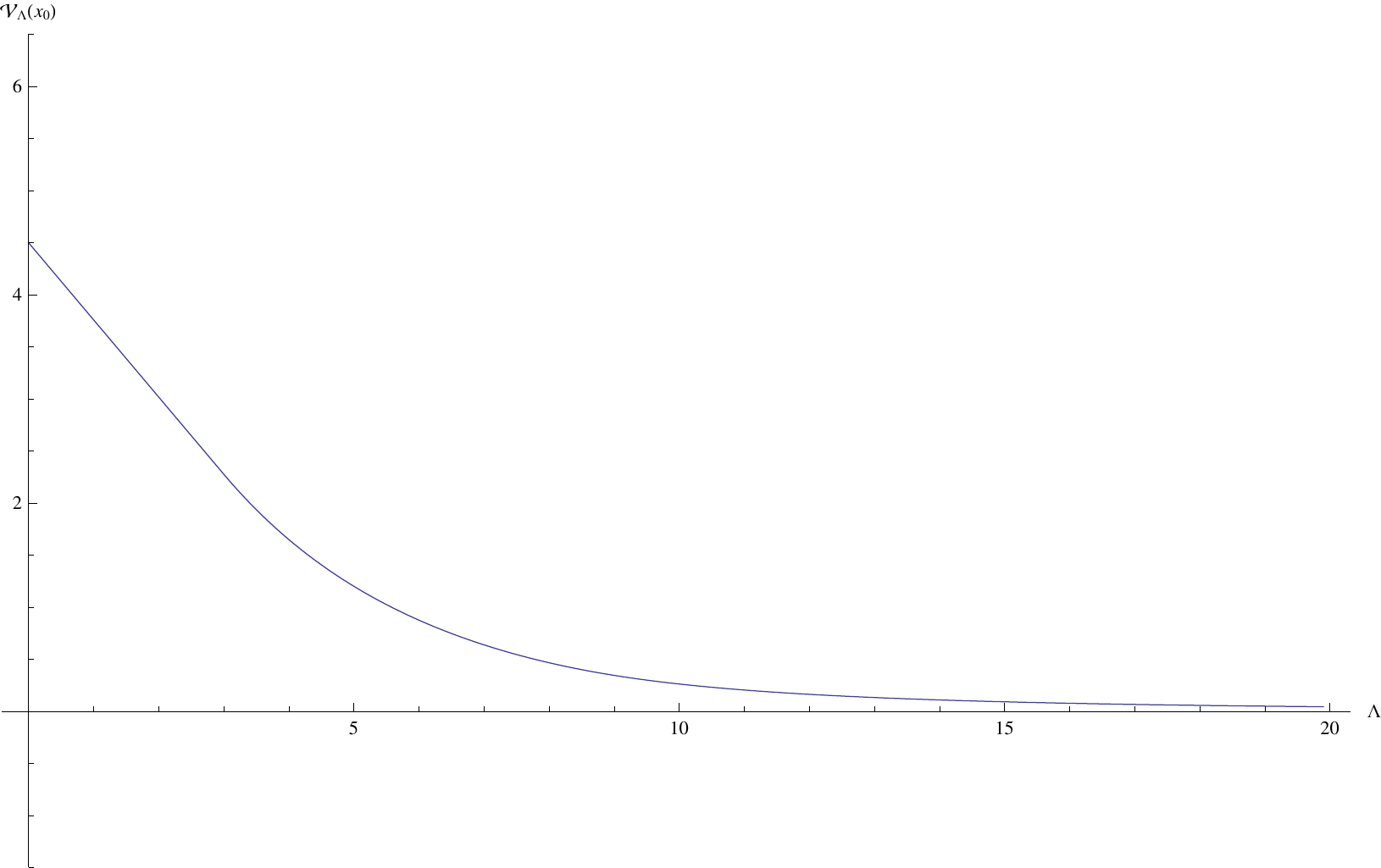}
   \end{subfigure}
      \caption{Do nothing.}\label{figfrontera}
\end{figure}

In the last case, let $T$ and $x_0$ have values such that $K_T$ lies above $\hat{\varPsi}(x_0)$. With such values the plot of $V_\Lambda(x_0)$ for different values of $\Lambda$ reflects that the minimum is also attained at $\infty$. This is due to the fact that there is no $b$ such that $\varPsi_{b}(x_0)=K_T$, see Figure \ref{figunfeasible}. The problem is infeasible so its optimal value is $-\infty$.\\
\begin{figure}[h!]
   \begin{subfigure}[b]{.49\linewidth}
      \includegraphics[width=1\textwidth]{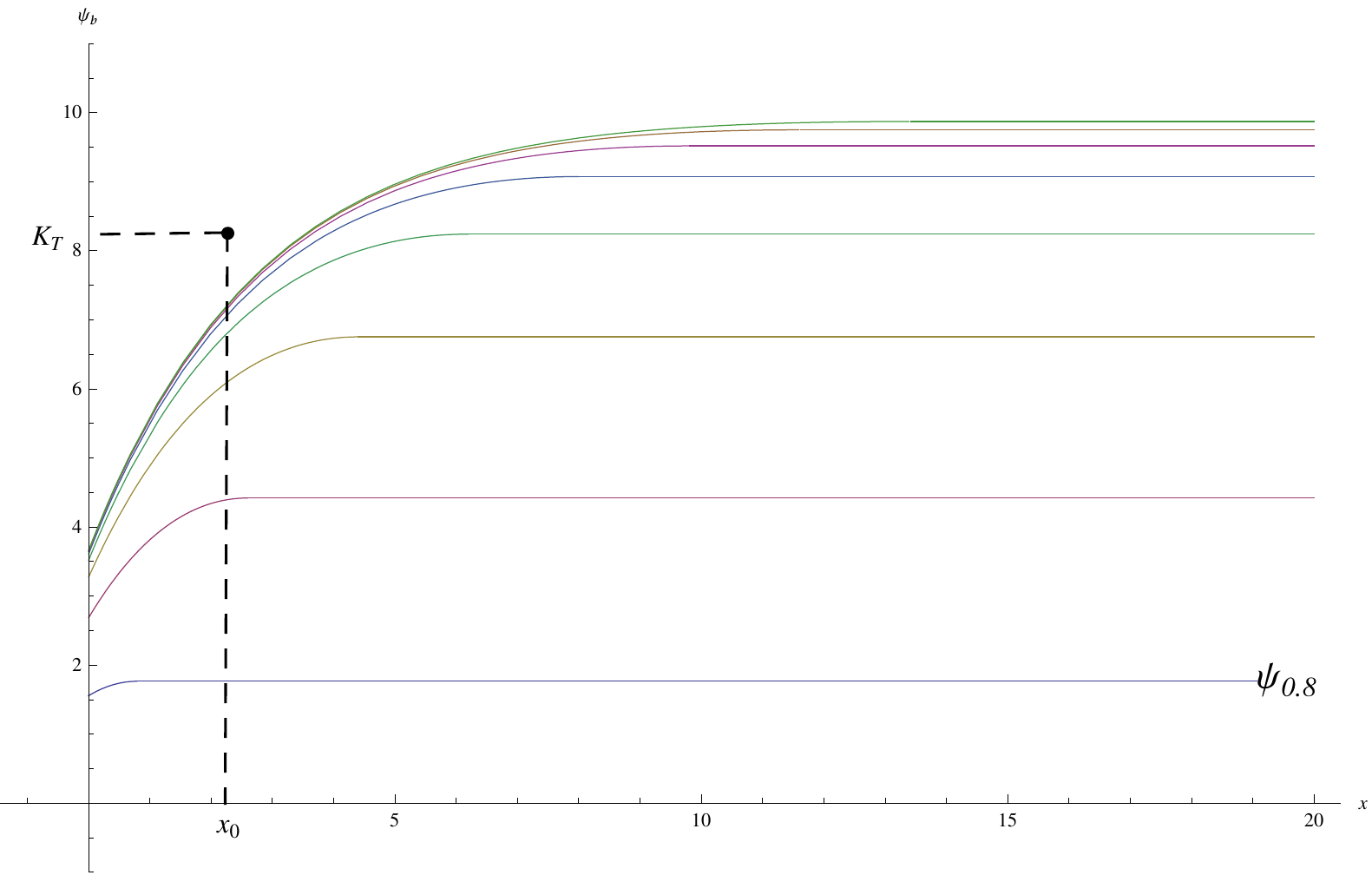}
   \end{subfigure}
\begin{subfigure}[b]{.49\linewidth}
      \includegraphics[width=1\textwidth]{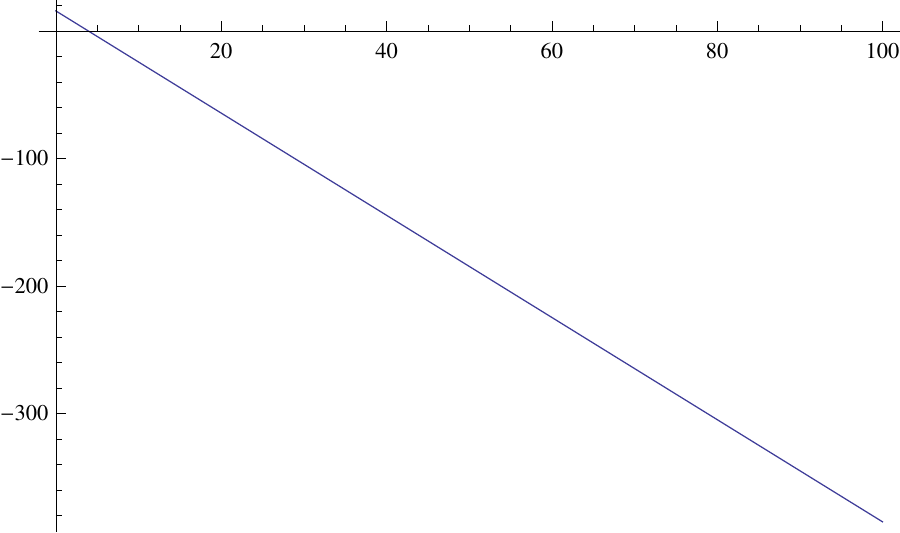}
   \end{subfigure}
      \caption{Problem unfeasible.}\label{figunfeasible}
\end{figure}
\clearpage
Figure \ref{Valfunctions} shows the value functions of both the unconstrained (Solid line) and the constrained problem (Dashed line) for $K_{20}$. In this case for $x<4.23$ the value function of the constrained problem equals $-\infty$.

\begin{figure}[h!]
  \centering
 \includegraphics[scale=.85]{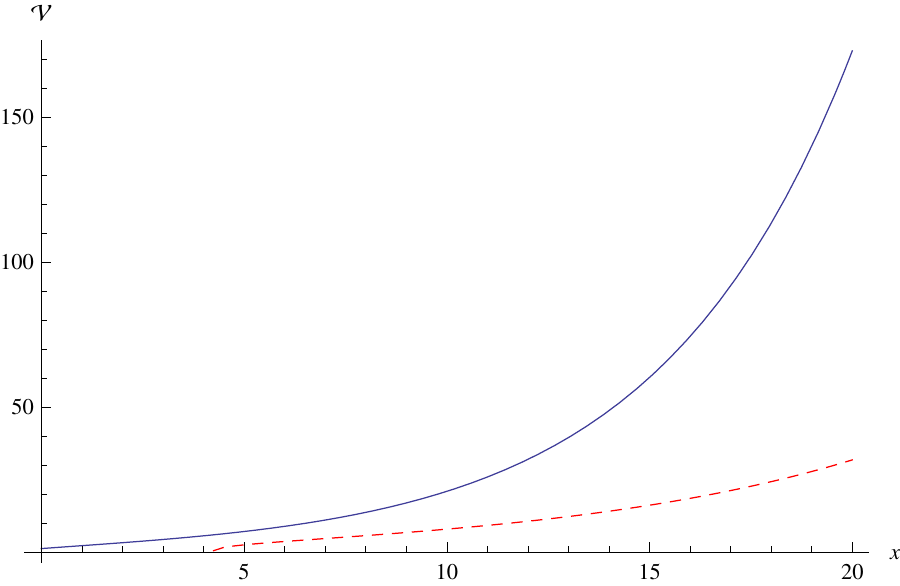}
  \caption{$V(x_0)$ for both the unconstrained and the constrained problem.}\label{Valfunctions}
\end{figure}

Finally, Figure \ref{figregions} shows how the solution to problem \eqref{P1} can be graphically characterized in terms of $(x_0,K_T)$, the horizontal line at level $\frac{1}{\delta}$, $\hat{\varPsi}(x)$ and $\varPsi_{b_0}(x)$.

\begin{figure}[h!]
  \centering
 \includegraphics[scale=.75]{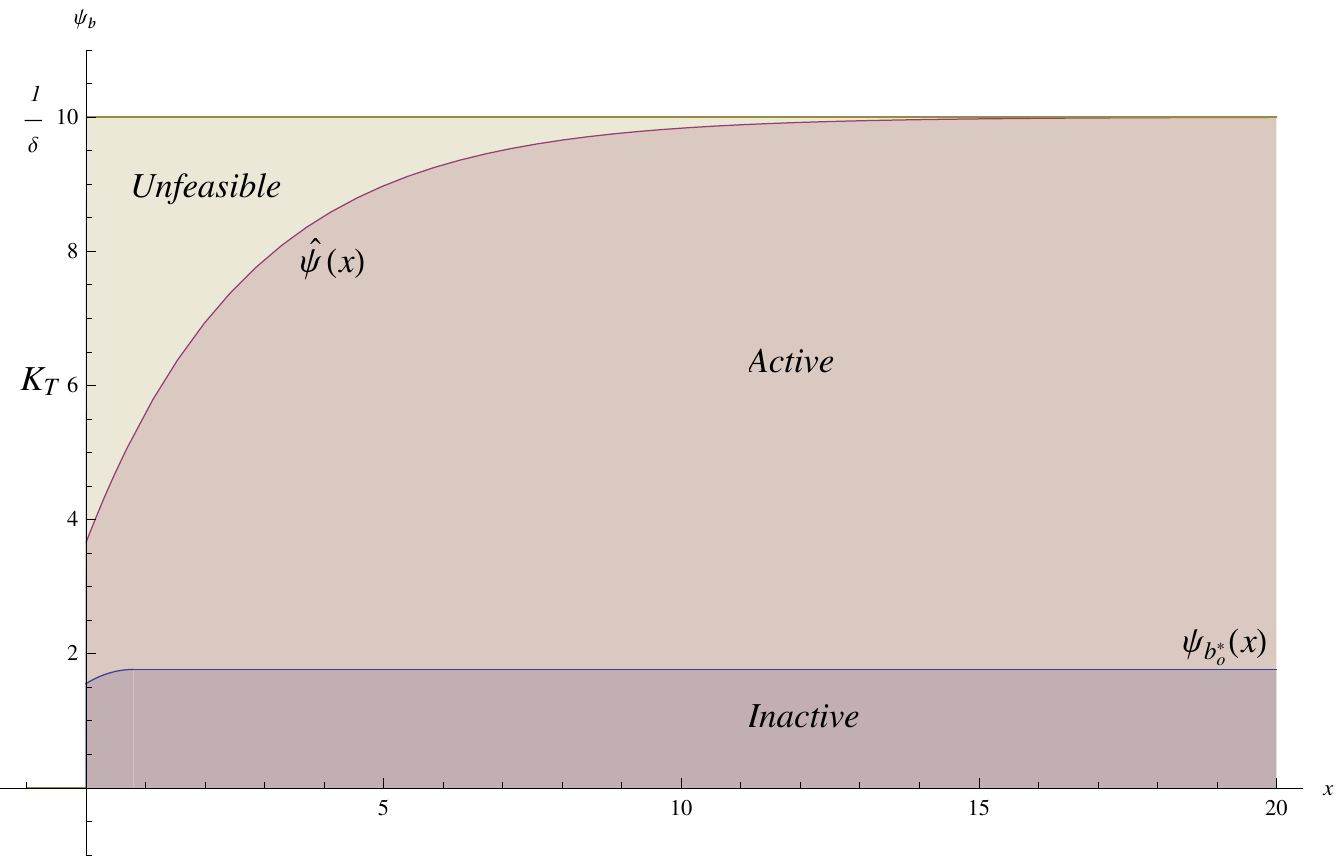}
  \caption{Solution description.}\label{figregions}
\end{figure}
\section{Conclusions and future work}

In the framework of the classical dividend problem there exists a trade-off between stability and profitability. Minimizing the ruin probability could lead to no dividend payment whereas maximizing the expected value of the discounted payments leads to a dividend payment trend for which ruin is certain regardless of the initial amount $x_0$. In this work we study a way to link the profits and the time of ruin derived from a dividend payment strategy $D$. We introduced a restriction that imposes a constraint on the time of ruin. Under exponentially distributed claim sizes distribution we succeed in solving the constrained problem using Duality Theory. Consider the problem with general claims distribution is part of future research. Ongoing research also involves different type of restrictions and time-dependent optimal strategies as well.

\appendix

\section{Auxiliary Lemmas}

\begin{lemma}\label{alphaandroot}
Let $r_2$ be the negative root of the characteristic polynomial (\ref{poli}). Then. $r_2+\alpha>0$.
\end{lemma}
\begin{proof}
\begin{align*}
r_2+\alpha &>0 \\
&\iff  \frac{-\alpha c + (\lambda +\delta) - \sqrt{([\alpha c - (\lambda + \delta)]^2 +4 c \alpha \delta)}}{2c} + \alpha >0\\
&\iff  2c \alpha > \alpha c - (\lambda +\delta) + \sqrt{([\alpha c - (\lambda + \delta)]^2 +4 c \alpha \delta)} \\
&\iff  c \alpha + (\lambda +\delta) > \sqrt{([\alpha c - (\lambda + \delta)]^2 +4 c \alpha \delta)} \\
&\iff  c \alpha + (\lambda +\delta) > \sqrt{([\alpha c + (\lambda +\delta)- 2(\lambda + \delta)]^2 +4 c \alpha \delta)} \\
&\iff  c \alpha + (\lambda +\delta) >\\
&\qquad \qquad  \sqrt{([\alpha c + (\lambda +\delta)]^2- 4(\alpha c + (\lambda + \delta))(\lambda + \delta) + 4(\lambda + \delta)^2 +4 c \alpha \delta)} \\
&\iff - 4(\alpha c + (\lambda + \delta))(\lambda + \delta) + 4(\lambda + \delta)^2 +4 c \alpha \delta <0\\
&\iff  4(\alpha c + (\lambda + \delta))(\lambda + \delta) > 4(\lambda + \delta)^2 +4 c \alpha \delta\\
&\iff  4\alpha c(\lambda + \delta)   >  4 c \alpha \delta\\
&\iff  4\alpha c\lambda  > 0.
\end{align*}
\end{proof}

\begin{lemma}\label{Phiincreasinb}
Let $\varPsi_b(x)$ be the solution of problem \eqref{A1}. For $x\geq0$ fixed, $\varPsi_b(x)$ is increasing in b.
\end{lemma}
\begin{proof}
The proof consists on calculate $ \frac{d\varPsi_b(x)}{db}$. For $x<b$
\begin{align*}
g(x)&:=\frac{d\varPsi_b(x)}{db}\\
&=\frac{(\alpha+r_1)(\alpha+r_2)r_2^2 e^{r_2b}\alpha\delta[r_1e^{r_1b}(\alpha+r_1)-r_2 e^{r_2b}(\alpha+r_2)]}{(\alpha\delta[r_1e^{r_1b}(\alpha+r_1)-r_2 e^{r_2b}(\alpha+r_2)])^2}e^{r_1x}\\
&-\frac{(\alpha+r_1)(\alpha+r_2)r_2 e^{r_2b}\alpha\delta[r_1^2e^{r_1b}(\alpha+r_1)-r_2^2 e^{r_2b}(\alpha+r_2)] }{(\alpha\delta[r_1e^{r_1b}(\alpha+r_1)-r_2 e^{r_2b}(\alpha+r_2)])^2}e^{r_1x}\\
&- \frac{(\alpha+r_2)^2 r_2^2 e^{r_2b}\alpha\delta[r_1e^{r_1b}(\alpha+r_1)-r_2 e^{r_2b}(\alpha+r_2)]}{(\alpha\delta[r_1e^{r_1b}(\alpha+r_1)-r_2 e^{r_2b}(\alpha+r_2)])^2}e^{r_2x}\\
&+\frac{(\alpha+r_2)^2r_2 e^{r_2b}\alpha\delta[r_1^2e^{r_1b}(\alpha+r_1)-r_2^2 e^{r_2b}(\alpha+r_2)] }{(\alpha\delta[r_1e^{r_1b}(\alpha+r_1)-r_2 e^{r_2b}(\alpha+r_2)])^2}e^{r_2x}.
\end{align*}
The numerator of this expression can be reduced to
\begin{align*}
&\alpha\delta r_2(\alpha+r_1)(\alpha+r_2) e^{r_2(b+x)}[r_1r_2e^{r_1b}(\alpha+r_1)\\
&\quad -r_2^2 e^{r_2b}(\alpha+r_2)-r_1^2e^{r_1b}(\alpha + r_1)+r_2^2e^{r_2b}(\alpha+r_2)]\\
&\qquad-\alpha\delta r_2(\alpha+r_2)^2 e^{r_2(b+x)}[r_1r_2e^{r_1b}(\alpha+r_1)\\
&\qquad \quad-r_2^2 e^{r_2b}(\alpha+r_2)-r_1^2e^{r_1b}(\alpha + r_1)+r_2^2e^{r_2b}(\alpha+r_2)]\\
&=\alpha\delta r_2r_1(\alpha+r_1)(\alpha+r_2)^2 e^{r_2(b+x)}e^{r_1b}[r_1-r_2]\\
&\qquad -\alpha\delta r_2r_1(\alpha+r_1)^2(\alpha+r_2) e^{r_2(b+x)}e^{r_1b}[r_1-r_2]\\
&=-\alpha\delta r_2r_1(\alpha+r_1)(\alpha+r_2) e^{r_2(b+x)}e^{r_1b}[r_1-r_2]^2>0.
\end{align*}
Now, for $x\geq b$, $\varPsi_b(x)=\varPsi_b(b)$ and therefore we must calculate
\begin{align*}
\frac{d\varPsi_b(b)}{db}=&g(b)+ \frac{(\alpha+r_1)(\alpha+r_2)r_2 e^{r_2b}r_1 e^{r_1b}}{\alpha\delta[r_1e^{r_1b}(\alpha+r_1)-r_2 e^{r_2b}(\alpha+r_2)]}\\
-&\frac{(\alpha+r_2)^2 r_2 e^{r_2b}r_2 e^{r_2b}}{\alpha\delta(e^{r_1b}r_1(\alpha+r_1)-e^{r_2b}r_2(\alpha+r_2))}- \frac{(\alpha+r_2)r_2 e^{r_2b}}{\alpha\delta}\\
=&g(b)+ \frac{(\alpha+r_1)(\alpha+r_2)r_1r_2 e^{(r_1+r_2)b}}{\alpha\delta[r_1e^{r_1b}(\alpha+r_1)-r_2 e^{r_2b}(\alpha+r_2)]}\\
-&\frac{(\alpha+r_2)^2 r_2^2 e^{2r_2b}-(\alpha+r_1)(\alpha+r_2)r_1r_2 e^{(r_1+r_2)b}+(\alpha+r_2)^2 r_2^2 e^{2r_2b}}{\alpha\delta(e^{r_1b}r_1(\alpha+r_1)-e^{r_2b}r_2(\alpha+r_2))}\\
=&g(b)>0.
\end{align*}
\end{proof}

\begin{lemma}\label{limit}
Let $x\geq0$ and $T\geq0$. If $K_T=\hat{\varPsi}(x)$ then $\Lambda\Big[\mathbb{E}_x[\int_0^{\tau^{b_\Lambda}}e^{-\delta s}ds] -K_T\Big]\rightarrow 0$ as $\Lambda\rightarrow \infty$.
\begin{proof}
From the proof of Proposition \ref{optimalpair} we must calculate $\Lambda\Big[\varPsi_{b_\Lambda}(x) -K_T\Big]$ as $\Lambda\rightarrow \infty$. Furthermore, from Proposition \ref{blambda1to1} and Lemma \ref{Phiincreasinb} we know the second term goes to $0$ as $\Lambda\rightarrow \infty$. Since $\Lambda\rightarrow \infty$ is equivalent to $b \rightarrow \infty$, we will calculate the limit of 
$$-\frac{[\varPsi_b(x)-K_T]^2}{\frac{d\varPsi_b(x)}{d\Lambda}}=-\frac{[\varPsi_b(x)-K_T]^2}{\frac{d\varPsi_b(x)}{db}\frac{db}{d\Lambda}} \text{ as }  b\rightarrow \infty.$$
Recall that
\begin{equation*}
\varPsi_b(x)= \frac{1}{\delta} + \frac{(\alpha+r_1)(\alpha+r_2)r_2 e^{r_2b+r_1x}-(\alpha+r_2)(\alpha+r_1)r_1 e^{r_1b+r_2x}}{\alpha\delta[r_1e^{r_1b}(\alpha+r_1)-r_2 e^{r_2b}(\alpha+r_2)]}
\end{equation*}
and that $K_T=\hat{\varPsi}(x)=\frac{1}{\delta}-\frac{\alpha+r_2}{\alpha \delta}e^{r_2 x}$. Therefore
$$(\varPsi_b(x)-K_T)^2=\left(\frac{(\alpha+r_2)r_2[e^{r_1 x}(\alpha+r_1)-e^{r_2 x}(\alpha+r_2)]}{(\alpha\delta[r_1e^{r_1b}(\alpha+r_1)-r_2 e^{r_2b}(\alpha+r_2)]}e^{r_2 b}\right)^2.$$
From Proposition \ref{blambda1to1} and Lemma \ref{Phiincreasinb} it can be derived that
\begin{align*}
\frac{d\varPsi_b(x)}{db}\frac{db}{d\Lambda}=&\frac{\delta \alpha (\alpha+r_1)^2(\alpha+r_2)^2[r_1-r_2]^3 e^{r_1b}e^{r_2(b+x)}}{(\alpha\delta[r_1e^{r_1b}(\alpha+r_1)-r_2 e^{r_2b}(\alpha+r_2)])^2}\\
& \qquad \cdot \frac{1}{e^{-r_2b}(r_1(\lambda+\delta)+ \alpha \delta)- e^{-r_1b}(r_2(\lambda+\delta)+\alpha\delta )}.
\end{align*}
Now,
\begin{align*}
-\frac{(\varPsi_b(x)-K_T)^2}{\frac{d\varPsi_b(x)}{db}\frac{db}{d\Lambda}}=
&-\frac{\left((\alpha+r_2)r_2[e^{r_1 x}(\alpha+r_1)-e^{r_2 x}(\alpha+r_2)]\right)^2e^{2r_2b}}{\alpha \delta (r_1-r_2)^3(\alpha+r_1)^2(\alpha+r_2)^2 e^{r_1b}e^{r_2b} e^{r_2x}}\\
&\cdot (e^{-r_2b}(r_1(\lambda+\delta)+ \alpha \delta)- e^{-r_1b}(r_2(\lambda+\delta)+\alpha\delta ))\\
=&-\frac{\left(\alpha+r_2)r_2[e^{r_1 x}(\alpha+r_1)-e^{r_2 x}(\alpha+r_2)]\right)^2}{\alpha \delta (r_1-r_2)^3(\alpha+r_1)^2(\alpha+r_2)^2 e^{r_2x}}\\
& \cdot (e^{-r_1b}(r_1(\lambda+\delta)+ \alpha \delta)- e^{(r_2-2r_1)b}(r_2(\lambda+\delta)+\alpha\delta )),
\end{align*}
from where one can conclude that the last expression goes to $0$ as $b\rightarrow \infty$ since $r_2<0<r_1$.
\end{proof}
\end{lemma}

\bibliographystyle{abbrv}
\bibliography{ref}

\end{document}